\documentclass[english,a4paper,10pt,final,reqno]{amsart}
\usepackage{standardpaketet}
\usepackage{float}
\usepackage{graphicx}
\usepackage{subfig}

   \newcommand{\al}{\alpha}
    
    \newcommand{\la}{\lambda}
    
    \newcommand {\bC} {\mathbb C}

\newtheorem{theo+}           {Theorem}
\newtheorem{prop+}           {Proposition}
\newtheorem{coro+}           {Corollary}
\newtheorem{lemm+}           {Lemma}
\newtheorem{conjecture}      {\bf Conjecture}
\newtheorem{defi+}           {Definition}

\newtheorem{not+}            {Notation}

\newtheorem{rema+}           {Remark}

\begin{document}
\title[Asymptotic distribution of zeros  for hypergeometric polynomials]{Asymptotic distribution of zeros of a certain class of hypergeometric polynomials}

\author{Addisalem Abathun }
\address{Department of Mathematics, Addis Ababa University and Stockholm University}
\email{addisaa@math.su.se}
\author{Rikard B\o gvad}
\address{Department of Mathematics, Stockholm University}
\email{rikard@math.su.se}
\date \today
\begin{abstract}We study the asymptotic behavior of the zeros of a family of  a certain class of hypergeometric polynomials ${}_{A}\text{F}_{B}\left[ \begin{array}{c} -n,a_2,\ldots,a_A \\ b_1,b_2,\ldots,b_B \end{array} ; \begin{array}{cc} z \end{array}\right]$, using the associated hypergeometric differential equation, as the parameters go to infinity.  We describe the curve complex  on which the zeros cluster, as level curves associated to integrals on an algebraic curve derived from the equation in the spirit of theorems in \cite{BBS}. In a certain degenerate case we make a precise conjecture, based and generalizing work in \cite{KBD0,KDPD1,LDJG}, and present experimental evidence for it.
\end{abstract}

\maketitle

\section{Introduction}
The generalized hypergeometric function ${}_{A}\text{F}_{B}$ is defined by the series
\begin{equation}
{}_{A}\text{F}_{B}\left[ \begin{array}{c} (a)_A \\ (b)_B \end{array} ; \begin{array}{cc} z \end{array}\right]=
{}_{A}\text{F}_{B}\left[ \begin{array}{c} a_1,a_2,\ldots,a_A \\ b_1,b_2,\ldots,b_B \end{array} ; \begin{array}{cc} z \end{array}\right]
=\sum_{k=0}^{\infty}\frac{\prod_{j=1}^{j=A}(a_{j})_{k}}{\prod_{j=1}^{j=B}(b_{j})_{k}}\frac{z^k}{k!}\label{eq1} \end{equation}
 where $(\alpha)_{k}=\alpha(\alpha+1)\ldots(\alpha+k-1)=\frac{\Gamma(\alpha+k)}{\Gamma(\alpha)}$ is the Pochhammer symbol. It has $A$ numerator parameters $a_1,a_2,\ldots,a_A,$ $B$ denominator parameters $b_1,b_2,\ldots,b_B,$ and one variable $z.$ Any of these quantities may be complex, but the denominator parameters must not be nonpositive integers.
   If one of the numerator parameters is a negative integer, say $a_1=-n,n\in\mathbb{N}$, the series terminates and the series in \eqref{eq1} reduces to a polynomial of degree $n$ in $z$, called a {\it generalized hypergeometric polynomial}.  (As general references for hypergeometric functions see \cites{A,S})

   We are interested in the asymptotics of the zeros of the hypergeometric polynomial \begin{equation}\label{hyp} p_n(z)={}_{A}\text{F}_{B}(-n,a_2(n),\ldots,a_A(n);b_1(n),\ldots,b_B(n);z)
   \end{equation}
   with complex parameters, dependent linearly on $n$, as $n\to \infty$. We show that under some assumptions on the asymptotic behavior of the parameters the zeros cluster on a finite  union of curves, which may be identified as level curves of certain harmonic functions. 
    We will do this by studying the asymptotic limit of the {\it root counting measure} $\mu_n=(1/n)\sum_{p_n(z)=0}\delta_z$, where $\delta_z $ is the Dirac measure associated to $z\in \mathbb C.$
     In fact the support of the limit $\mu=\lim_{n \to \infty}\mu_n$ (which we know exists for subsequences, by a compactness argument) is given as the "singular set" of a certain subharmonic function, which is locally obtained as a maximum of certain integrals of algebraic functions associated to the differential equation.
      On the basis of graphical evidence we conjecture a more precise expression as a global maximum in a special situation and give a generalization of previous results.

  \section{Review of earlier works on zeros of hypergeometric functions}

 The problem of locating zeros of Gauss
hypergeometric polynomials ${}_{2}\text{F}_{1}(-n,b;c;z)$ when $b$ and $c$ are arbitrary parameters has not been completely solved, though it has been studied by several authors. Even when $b$ and $c$ are both real, one needs to impose additional conditions on $b$ and $c$, see \cite{LDJG,KBD0,KDPD1}. We will now describe certain of these results. 
\\Driver and Duren \cite{KDPD1} show
that the zeros of the hypergeometric polynomials $F(-n, kn + 1; kn + 2; z), \text{ where }
k, n \in N$, cluster on the loop of the lemniscate
$\{z: |z^k(1 - z)| = \frac{k^k}{(k + 1)^{k+1}}$, with $ Re(z) > \frac{k}{k + 1}\}$
 as $n\rightarrow \infty.$


 Duren and Guillou \cite{LDJG}, using an independent and direct proof, extend this result to the case where $k$ is arbitrary positive real number. They have obtained similar results for arbitrary $k>0$.
Moreover, they observed that every point of the curve is a cluster point of zeros.

 Boggs and Duren \cite{KBD0} extend the discussion of  \cite{LDJG} further to a more general case.
They prove that for each $k > 0$ and $l \geq 0$, the zeros of 
$${}_2F_1(-n, kn + l + 1; kn + l + 2; z)$$
cluster on the loop of the lemniscate
$${z: |z^k(1 - z)| = \frac{k^k}{(k + 1)^{k+1}}},  Re(z) > \frac{k}{k + 1}, \text{ as } n\to \infty.$$
They also proved the conjecture in Duren and Guillou \cite{LDJG} about zero-free regions of hypergeometric polynomials. This case is very interesting since its Euler integral representation can be approximated by either Laplace integral method or a method of steepest ascent or descent.

In this paper we in particular give a different explanation of the appearance of these lemniscates, as well as a proof of part of these results, as well as an extension to generalized hypergeometric functions.

 \section{The Hypergeometric differential equation}
 \subsection{Basic notations and Definitions}
 We will make extensive use of the differential equation that hypergometric polynomials satisfy, which we will describe in terms of differential operators. The following definition is standard.
\begin{definition}
An operator $$T=\sum\limits_{i=0}^{k}Q_{i}(z)\frac{d^i}{dz^i}$$ is called exactly solvable  if ${\rm deg}\  Q_i(z)\leq i $ and there exists at least one i such that ${\rm deg}\ Q_i(z)=i$.
It is called non-degenerate if ${\rm deg}\ Q_k(z)=k$.
\end{definition}
\begin{not+}
We denote the space of exactly solvable differential operators of order at most k by $\text{ES}_{k}$.
 \end{not+}
 Consider a parameterized curve $\text{T}_{\lambda} : \textbf{C}\to ES_K$ ,
i.e.
\begin{equation}
 T_\lambda=\sum\limits_{i=0}^{k}Q_{i}(z,\lambda)\frac{d^i}{dz^i}.\label{opo}
\end{equation}
Here we assume that $Q_i(z,\lambda)$ are polynomials and then ${\rm deg}_zQ_i(z,\lambda)\leq i$  by the condition of exactly solvability. We will furthermore only consider the case when ${\rm deg}_\lambda Q_i(z,\lambda)\leq k-i,$ and then call $T_\lambda$ a {\it spectral pencil}.
\begin{definition}
For each $\lambda \in \mathbb{C}$, a function $f$ which satisfies the equation $T_\lambda(f)=0$ is called a generalized eigenfunction belonging to the eigenvalue $\lambda$.
 \end{definition}
 A particular case is ordinary eigenfunctions of the operator $T$, that is when $T_{\lambda}=\sum_{i=1}^{k}Q_k\frac{d}{dz^i}  -\lambda$.  By definition, a generalized eigenfunction then corresponds to a solution of $Tf-\lambda f =0,$ that is, exactly an eigenfunction of $T$.
\\ A special case of a spectral pencil is the homogenized case, when
\begin{equation}
\label{eq:hom}
T_\lambda=\sum\limits_{i=0}^{k}Q_{i}(z)\lambda^{k-i}\frac{d^i}{dz^i}.
\end{equation}

In this case  $Q_i(z)=a_{i,i}z^i + a_{i,i-1}z^{i-1} + \ldots$ is a polynomial of  degree less than or equal to $ i$. To each spectral pencil (\ref{opo}) with $Q_i(z,\lambda)=R_i(z)\lambda^{k-i}+..$ we have the {\it associated
homogeneous spectral pencil}  $$T^h_\lambda=\sum\limits_{i=0}^{k}R_{i}(z)\lambda^{k-i}\frac{d^i}{dz^i}.$$
Furthermore, we may to the pencil associate a {\it characteristic equation}, whose roots will be important. If 
$R_i(z)=\sum_{j=0}^ia_{ij}z^j$, $i=0,...,k$, then the characteristic equation is
$$
\sum_{i=0}^ka_{ii}\lambda^{k-i}=0.
$$


  \subsection{The hypergeometric differential operator}
  The Gauss hypergeometric function
 ${}_{2}\text{F}_{1}[a,b;c;z]$ is a solution to Gauss hypergeometric differential equation
  \begin{equation*}
     z(1-z)\frac{d^2y}{dz^2} + [c-(a+b+1)z]\frac{dy}{dz}-aby=0,
     \end{equation*}
    Using $\triangle :=z\frac{d}{dz}$, the Gauss equation can also be written as
   \begin{equation*}\label{diff0}
 \triangle(\triangle +c-1)y=z(\triangle +a)(\triangle +b)y
     \end{equation*}


A corresponding result also holds for the generalized hypergeometric functions. The function
 ${}_{A}\text{F}_{B}[(a)_A;(b)_B;z]$
             satisfies the linear differential equation
  \begin{multline}\label{diff}
 \bigg(
 z\frac{d}{dz}\left(z\frac{d}{dz}+b_1-1\right)\left(z\frac{d}{dz}+b_2-1\right)\ldots\left(z\frac{d}{dz}+b_B-1\right)
 \\-z\left(z\frac{d}{dz}+a_1\right)\left(z\frac{d}{dz}+a_2\right)\ldots\left(z\frac{d}{dz}+a_A\right)\bigg)y=0
 \end{multline}
 This equation is called the generalized hypergeometric differential equation.


Equation \eqref{diff} can be written as
   \begin{equation}\label{diff2tri}
      {}_{A}T_{B}y=0,
      \end{equation}
    where
  \begin{equation}\label{oper1}
      {}_{A}T_{B}=
      \frac{d}{dz}\prod_{i=1}^{B}(z\frac{d}{dz}+b_i-1)-\prod_{i=1}^{A}(z\frac{d}{dz}+a_i)
      \end{equation}
      is the {\it hypergeometric differential operator}(we suppress the parameters in the notation).
       \begin{proposition}
       Let $A=B+1$. Then  ${}_AT_B$ is non-degenerate exactly solvable.
       \end{proposition}
       \begin{proof}
       We can write
       \begin{equation*}\label{oper11}
      {}_{A}T_{B}=
     \frac{1}{z} \triangle \prod_{i=1}^{B}(\triangle +b_i-1)-\prod_{i=1}^{B+1}(\triangle +a_i).
      \end{equation*}
     Simplifying the right hand side we obtain 
     
      \begin{equation}\label{oper12}
      {}_{A}T_{B}=
     ( \frac{1}{z}-1)\triangle^{B+1}+\frac{(e_{1}^{\tilde{b}}-e_1^{a}z)}{z}\triangle^B+...+\frac{(e_{B}^{\tilde{b}}-e_B^{a}z)}{z}\triangle-e_{B+1}^a.
      \end{equation} 
      Here
      $$e_i^{a}:=\sum_{1\leq j_1<\ldots < j_i\leq A}a_{j_1}\ldots a_{j_i} \quad (i=1,2,\ldots,B+1),$$
      are the elementary symmetric functions in $a_1,...,a_{B+1}$, and 
      $e_i^{\tilde{b}}$ similarly are the
      elementary functions in $ \tilde{b}_{j}=b_{j}-1, j=1,2,\ldots B.$
      Now note that $$\triangle^i=z^i\frac{d^{i}}{dz^{i}}+...+z\frac{d}{dz}=\sum_{j=1}^{i}a_j(i)z^j\frac{d^{j}}{dz^{j}},$$
      for some integers $a_j(i)$.
       Then further simplifying the right hand side of \eqref{oper12} we obtain
      \begin{equation*}\label{oper12:2}
      {}_{A}T_{B}=
      z^B(1-z)\frac{d^{B+1}}{dz^{B+1}}+\sum_{j=1}^{B}(C_j-D_jz)z^{j-1}\frac{d^{j}}{dz^{j}}-e_{B+1}^a,
      \end{equation*}
    
     where \begin{equation}
       \label{eq:coeff1}
C_j-D_jz=a_j(B+1)(1-z)+\sum_{i=1}^{B+1-j}(e_{i}^{\tilde{b}}-e_i^{a}z)a_j(B+1-i).
\end{equation}

       The degree of the coefficient polynomial $Q_{j}(z)$ of $\frac{d^j}{dz^j}$ for all $j=1,2,\ldots,B+1$ is then less than or equal to $j$ and in fact $\text{ deg }Q_{B+1}=B+1$. Hence ${}_{A}T_{B}$ is an exactly solvable non-degenerate operator.

       \end{proof}
 \begin{proposition}
 Let $a_i=\al_i n+c_i$, where $\al_i\in\bC \setminus \{0\}$, and $c_i\in\bC,$  $i=1,2,\ldots,A$  and  $b_j=\beta_i n+d_i+1$,  $\beta_i,\ d_i \in\bC$,   $j=1,2,\ldots,B$, and assume that $A=B+1.$ Using $\lambda=n$ as spectral parameter,  $T_n={}_AT_B$ is a spectral pencil as in \eqref{opo} for large $n$. If $c_i=0$ and $d_i=0$ for all $i$ then this pencil is homogenized. In general, its characteristic equation is
  $$\prod\limits_{j=1}^{A}\big( 1+\lambda \al_i\big) = 0.$$
 \end{proposition}

 \begin{proof}
 Using the notation from the proof of the preceding proposition, it is clear that
$$e_i^{a}:=e_i^{\alpha}n^i+\text{ lower terms in }n....,$$
and 
      $e_i^{\tilde{b}}=e_i^{\beta}n^i+\text{ lower terms in }n....$.
      Hence by (\ref{eq:coeff1})
     \begin{equation}
(C_j-D_jz)z^{j-1}=(e_{B+1-j}^{\beta}-e_{B+1-j}^{\alpha}z)z^{j-1}n^{B+1-j}+\text{ lower terms in }n,
\end{equation}
and this implies that the pencil satisfies the degree condition in $n$.
      
Taking the associated homogeneous equation means that we drop the lower terms in $n$, and then the  characteristic equation is $e^{\al}_{B+1}n^A+e^{\al}_{B}n^{A-1}+\ldots+e^{\al}_{1}n+1=0,$ which implies the rest of the proposition.

\end{proof}
We will now make a further non-degeneracy assumption, adapted from a corresponding condition in \cite[Prop. 5]{BBS}. Recall that we are interested in the situation when $a_1=-n$ or equivalently $\alpha_1=-1$. In this case the solutions to the characteristic equation are 
$$\la_1=1 ,\la_2=\frac{-1}{\al_2},\ldots,\la_A=\frac{-1}{\al_A}.$$

\begin{definition}
A hypergeometric spectral pencil $T_\lambda$ is called of {\it general type} if the characteristic equation has simple roots, i. e. no two $\alpha_i,  i=2,...,A$ coincide and they are all distinct from $-1$, and furthermore no $-\alpha_i,  i=2,...,A$ is contained in the interval $]-\infty, 1]$ on the real axis.
\end{definition}

This implies that the results of \cite{BBS} hold (even though the definition of general type in that paper is slightly different), as can be seen by Proposition 5 of that paper. We will discuss them in the next section.

\subsection{The Cauchy Transform }
In this section we will study zeros of polynomials in the form of probability measures.
\begin{definition}

Let $p_n(z)$ be a polynomial of degree $n$ with zeros $\zeta_1,\zeta_2,\ldots,\zeta_n$  listed according to multiplicities. Define
\begin{equation}\label{mu}
\mu_{n}=\frac{1}{n}\sum_{\nu=1}^{n}\delta_{\zeta_\nu}
\end{equation}
where $\delta_{\zeta_\nu}$ is a Dirac measure; we call $\mu_n$  a root counting measure of $p_n$. Clearly $\mu_n$ is a probability measure of total mass $1$.
\end{definition}

It is possible to recover the logarithmic derivative $\frac{p'_n}{np_n}$ as the Cauchy transform of $\mu_n$ using the following definitions.

\begin{definition}
 Let $\mu$ be a finite, compactly supported measure. The Cauchy transform $C_\mu$ of $\mu$ is the function
\begin{equation}\label{ch}
C_\mu(z)=\int\frac{d\mu(\zeta)}{\zeta-z}.
\end{equation}
 
\end{definition}
The integral converges for all z, for which the Newtonian potential
\begin{equation*}
U_|\mu|(z)=\int\frac{d|\mu|(\zeta)}{|\zeta-z|}
\end{equation*}
is finite. Since $\frac{1}{|z|}\in L_{loc}^1$, Fubini's theorem gives that $U_{|\mu|}$ converges almost everywhere and $C_\mu$ is in $L^1_{loc}.$ Moreover, $C_\mu$ is analytic on $\bC \setminus \text{ Supp }\mu.$

For a sequence of generalized eigenpolynomials of a non-degenerate spectral pencil, such that the corresponding root measure converges, one may from \cite{BBS,BB} conclude that the Cauchy transforms of the root measures converges to a $L^1_{loc}$-function satisfying an algebraic equation. In the present case this algebraic equation has a particularily simple  form.
 \begin{theorem}
\label{cauchy}
Let $\{p_n\}$  be a family of hypergeometric polynomials 
$$p_n(z) = {}_{A}\text{F}_{B}\left[ \begin{array}{c} (a) \\ (b) \end{array} ; \begin{array}{cc} z \end{array} \right],$$ where 
 $a_1=-n=\alpha_1n$ and $a_i=\al_i n+c_i$, for some $\al_i, c_i\in\bC \setminus \{0\}$, and  $b_j=\beta_i n+d_i+1$, for some $\beta_i,d_i \in\bC$,   $j=1,2,\ldots,B$, are the parameters and we assume that $A=B+1.$ 
 Let $\mu_n$ be the root measure of $p_n$. Assume that $\alpha_i\neq 0,\ i=1,...,A$. If the subsequence  $\mu_{n_i}, \ i=1,...$ of root measures converges weakly to the finite measure $\mu$ the following holds:
 \begin{itemize}
 \item[(i)]  $C_{\mu_{n_i}}\to C_\mu$ almost everywhere in $\bC$
\item[(ii)] $w=C_\mu(z)$ satisfies the algebraic equation 
\begin{equation}
\label{eq:alg}
A(z,w)=\prod_{i=1}^{i=A}(zw+\alpha_i)-w\prod_{i=1}^{i=B}(zw+\beta_i)=M(zw) - wN(zw)=0.
\end{equation}

  \end{itemize}
  
  Furthermore if the pencil is of general type there always exist subsequences of $\mu_n,\ n=1,...$ that do converge to finite measures.
    \end{theorem}
    
Note that $C_\mu$ does not have to be a continuous function. In fact the locus of non-continuity is more or less the support  of the limit measure $\mu$, as we will see. We will give a sketch of the proof of the theorem, and for this we need a lemma, describing the non-linear differential equation satisfied by the logarithmic derivatives of the polynomials in the sequence. With the help of this description we can 
transform the linear differential equation for $p_n$ into a non-linear equation for the logarithmic derivative.

\begin{lemma}
Let $p_n(z)=C\prod\limits_{j=1}^{n}(z-\zeta_j),$ where $C$ is a non zero constant. Let $\mu_n$ be the root counting measure of $p_n.$ The Cauchy  transform is
$$ C_{\mu_{n}}(z)=\frac{1}{n}\sum\limits_{\nu=1}^{n}\frac{1}{z-\zeta_{\nu}}=\frac{p_n'}{np_n}.$$
Furthermore
\begin{equation}\label{cau3}
 \frac{p^{(i)}_n}{p_n}= n^iC^i_{\mu_n}+G(n,C_{\mu_n},C'_{\mu_n},\dots,C^{(i-1)}_{\mu_n}),
 \end{equation}
 where
 $G$
 is a polynomial, of degree $i-1$ in the variable $n.$
 \end{lemma}
\begin{proof}  The first part is Cauchy's theorem. The second follows by induction, starting with $C_{\mu_{n}}(z)=\frac{p_n'}{np_n}, $ and then using
 $$\frac{d}{dz}\bigg(\frac{p_n^{(i)}}{p_n}\bigg)=\frac{p_n^{(i+1)}}{p_n}- \frac{p_n^{(i)}}{p_n}\frac{p_n'}{p_n}.$$
 \end{proof}

 \begin{proof}[ Proof of the theorem.]
     The last part of the theorem, the existence of subsequences which are weakly convergent follows from the fact, proved in the main theorem of \cite{BBS}, that polynomial eigenfunctions to pencils of general type have zero-sets contained in a compact set that is independent of $n$. The results in loc.cit. are formulated for the case of a homogeneous pencil,  but the proof works in our situation, mutatis mutandis, and gives a result that only depends on the associated homogenized spectral pencil. 
     
     Part (i) is an immediate consequence of the definition of weak convergence of measures. Part (ii) follows from Proposition 2 in \cite{BBSc},(corrected version of proposition $3$ in \cite{BBS}),  noting the fact that the characteristic equation of the pencil has  highest coefficient $\prod\alpha_i\neq 0$ and constant term $1$. We then only need to identify the algebraic equation in question. We will at the same time give a heuristic illustration of the proof in \cite{BBS} using a stronger condition than in      
\cite{BBS,BBSc}.

    In the previous section we have seen that $p_n(z)$ satisfies the differential equation ${}_AT_Bp_n(z)=0$. That is
    $$z^B(1-z)p_n^{(B+1)}+\sum_{k=1}^{B}(e_k^{\beta}-e_k^{\al}z)n^{B+1-k}z^{k-1}p_n^{(k)}+e_{B+1}^\al n^{B+1}p_n=0.$$
    Dividing by $p_n$ we get
    $$z^{B}(1-z)\frac{p_n^{(B+1)}}{p_n}+\sum_{k=1}^{B}(e_k^{\beta}-e_k^{\al}z)z^{k-1}n^{B+1-k}\frac{p_n^{(k)}}{p_n}+n^{B+1}e_{B+1}^\al =0$$
    From equation \eqref{cau3} we can then write this as as
    $$z^{B}(1-z)(n^{B+1}C_{\mu_n}^{B+1}+G_{B+1})+\sum_{k=1}^{B}(e_k^{\beta}-e_k^{\al}z)z^{k-1}(n^{B+1}C_{\mu_n}^k+G_k)
    +n^{B+1} e_{B+1}^\al = 0, $$
    where $G_k,\ k=1,...B+1$ involves lower order terms in $n$ and derivatives of $C_{\mu_n}.$
    Now assume  that $C_{\mu_n}$ and all its derivatives of order $B$ are bounded(this is the strong condition, alluded to above, which is circumvented in the quoted sources). Divide the equation by $n^{B+1}$, and as $n\to \infty$ we then obtain
    $$z^B(1-z)C_{\mu}^{B+1}+\sum_{k=1}^{B}(e_k^{\beta}-e_k^{\al}z)z^{k-1}C_{\mu}^k +e_{B+1}^\al =0. $$
    Multiplying the above equation by $z$ we get
    $$\big(zC_\mu\big)^{B+1}-z\big(zC_\mu\big)^{B+1}+\sum_{k=1}^{B}e_k^\beta \big(zC_\mu\big)^k-z\sum_{k=1}^{B}e_k^\al \big(zC_\mu\big)^k+ze_{B+1}^\al=0.$$
    Thus
    $$\big(zC_\mu\big)^{B+1}+\sum_{k=1}^{B}e_k^\beta \big(zC_\mu\big)^k -z\bigg(\big(zC_\mu\big)^{B+1}+\sum_{k=1}^{B}e_k^\al \big(zC_\mu\big)^k+e_{B+1}^\al\bigg)=0.$$
    Setting $w=C_\mu$, this can be written as
\begin{equation}
\label{eq:alg2}
zw\prod_{i=1}^{B}(zw+\beta_i)-z\prod_{i=1}^{A}(zw+\alpha_i)=0,
\end{equation}
which by dividing by $-z$ gives the equation in the theorem.
    \end{proof}

 In general \eqref{eq:alg} defines $w$ as an algebraic function of $z$(of a quite special kind, by the way). A very natural question here is to ask how one can  choose the parameters $a_i$ and $b_i$ so that $M(zw)$ and $N(zw)$ have common roots, since then the algebraic equation has rational solutions. We use the notation of the theorem.

 \begin{proposition}
   \label{prop3}
Let $p_n(z) = {}_{A}\text{F}_{B}\left[ (a)_A ;\\ (b)_B;z \right].$ Assume that $\beta_i=\alpha_{i+1}$ for  $1\leq i\leq B$. Then \eqref{eq:alg} has the solutions 
$$w=f_1(z)=\frac{1}{z-1},\ w=f_i(z)=\frac{-\alpha_i}{z},\ i=2,...,A.$$  
\end{proposition}
\begin{proof} 
Under the assumption in the proposition, \eqref{eq:alg2} becomes
$$A(x,w)=zw(zw+\al_2)\ldots(zw+\al_A)-z(zw-1)(zw+\al_2)\ldots(zw+\al_A)=0,$$
This implies
 $$(zw+\al_2)\ldots(zw+\al_A)(zw(1-z)+z)=0,$$
 which clearly has the given solutions.
  \end{proof}
  \subsection{The asymptotic measure}
  Now we want to use the preceding results on the Cauchy transform to get information on the support of the possible limits of root counting measures, as in the Theorem in the preceding section. The key is to consider the logarithmic potential of $\mu$.
  
  \begin{definition} If $\mu$ is a probability measure, 
  the logarithmic potential is defined as the convolution
  $$
  L_\mu(z):=\int\log\vert z-\zeta\vert d\mu(\zeta).
  $$
  \end{definition}
 
 It will be  a subharmonic $ L_{loc}^1$ -function, harmonic outside the support of $\mu$. Furthermore $\partial L_\mu(z)\partial z=C_\mu(z)$. The fact that $C_\mu(z)$ satisfies an algebraic equation, then forces strong restrictions on the behavior of $L_\mu$.
 
 Let $U\subset\bC\setminus \{0,1\}$  be simply-connected open set, and $p\in U$. Consider the algebraic equation $A(z,w)=0$, belonging to the situation in Theorem \ref{cauchy}. Outside the branch points this equation defines an algebraic function with $A$ branches $w=f_i(z)$. To them we may associate real harmonic functions ${H_i}$, defined by 
 $$H_i(z)=Re\bigg[\int_{p}^{z}f_i(s)ds\bigg] \qquad z\in U, 1\leq i\leq A.$$
In this situation the following result is an immediate consequence of the main result of \cite{BB,BBB} (see also \cite{BR1}) on subharmonic functions whose derivative satisfies an algebraic equation.
 \begin{theorem}
 \label{levelsets0}
 If $\lim \mu_{n_i}=\mu$ for a subsequence of root counting measures $\mu_n$, then $\mu$ will have  support in $U$ on a finite union $K$ of parts of  level curves to the differences
 $H_i-H_j$ where $1\leq i\neq j\leq A.$ Furthermore in any sufficiently small disk $D$ that only contains points from the level curve  $H_i=H_j+C$.  one has
 $L_\mu(z)=\text{Max}\{ H_i(z),H_j(z)+C\}+D$, for some $D\in \mathbb C$(that only depends on $U$).
 \end{theorem}
  
Now we return to the degenerate situation, when the solutions to the algebraic equations are rational functions.  
Let $U\subset\bC\setminus \{0,1\}$  be simply-connected open set, and $p\in U$. Then the harmonic functions ${H_i},$ are $$H_i(z)=Re\bigg[\int_{p}^{z}f_i(s)ds\bigg] \qquad z\in U, 1\leq i\leq A.$$
Letting $\al_i =\al_i^x + \al_i^y i$ we obtain
\begin{eqnarray*}
\label{harm3}
H_1(z)&=\log|1-z|-\log |1-p|,\\
H_i(z)&=-\al_i^x \log |z|+ \al_i^{y} Arg z + \al_i^x\log |p|-\al_i^y Arg p ,
&\qquad 2\leq i\leq A
\end{eqnarray*}

By \ref{levelsets0} the asymptotic measure $\mu$ will have  support in $U$ on a finite union of parts of  level curves to the differences
 $H_i-H_j$ where $1\leq i\neq j\leq A.$ 
  \begin{theorem}
  \label{levelsets}
 Under the conditions of Proposition \ref{prop3}, a limit $\mu $ of a convergent subsequence of root counting measures will have support consisting of a union of parts of level curves  to the functions
 \begin{equation}\label{harm2}
 (\al^x_j-\al^x_i)\log|z| +(\al^y_j-\al^y_i)Arg z,\quad  2\leq i\neq j \leq A,
\end{equation}
and
  \begin{equation}\label{harm1}
   \log|1-z|+\al_i\log|z|-\al^y_iArg z, \quad\text{ for }2\leq i\leq A .
 \end{equation}
\end{theorem} 
Note that if  the hypergeometric pencil is of general type  (i.e. $-\alpha_i\notin ]-\infty,1], \ i=2,...,A$)
we know that there exist convergent subsequences, by a compactness argument. Also note that it is possible to recover $\mu$ as the Laplacian of $L_\mu$(see loc.cit).The above result is {\it local}, in the sense that we do not know which curves occur in the curve complex. We now want to explore a possible {\it global} descriptions of the support.

  The branch points of the curve defined by $A(z,w)=0$ are exactly the points where $ f_i(z)=f_1(z)\iff z=p_i=\frac{\al_i}{\al_i+1}.$ Through each $p_i$ there is a unique level curve $H_1(z)=H_i(z)+(H_1(p_i)-H_i(p_i))$.
 Now define $\tilde H_i(z)=H_i(z)+(H_1(p_i)-H_i(p_i))$, so that the level curve through $p_i$ is given by $H_1(z)=\tilde H_i(z)$.


\begin{example} Let
$p_n(z)={}_{2}\text{F}_{1}\left[ \begin{array}{c} -n, kn+1, \\ kn+2 \end{array} ;
 \begin{array}{cc} z \end{array}\right]$ for $k>0$, and let $\mu_n$ be the root measure of $p_n$ and assume $\mu_n \to \mu$ \text{as} $n \to \infty$(possibly for a subsequence). Then $p_n$ satisfies the equation
 $$(\bigtriangleup+kn+1)(\bigtriangleup-z(\bigtriangleup-n))p_n=0.$$
 Letting $n\to\infty$ we have the equation for the Cauchy transform
 $$A(z,w)=(zw+k)(zw-z(zw-1)=0,$$
 with solutions
 $$w=\frac{-k}{z} \qquad or \qquad w=\frac{1}{z-1}.$$
 The branch point will be $p=\frac{k}{k+1}$, and
hence $H_1(z)=\log |1-z|$ and $\tilde H_2(z)=-k\log |z|+(log  \frac{k^k}{(k+1)^{k+1}}.$ Hence the level curve to $H_1-\tilde H_2=0$ through the branch point becomes  
 becomes the lemniscate
 $$|z^k(1-z)|= \frac{k^k}{(k+1)^{k+1}}.$$ In this case there are strong asymptotic results, using the integral representation of hypergeometric functions, by  Duren and Guillou \cite{LDJG}, that show
  that the zeros of the polynomial $p_n$ cluster on the  loop of lemniscate where $Rez>\frac{k}{k+1}$.
 \end{example}
 Based on this, Theorem  \ref{levelsets} and experimental evidence(see next section) we conjecture that the zeros cluster along the level curve of $H_1-H_2$ through the branch point even for more general $k.$
 \begin{conjecture}
 Let $\al = \eta + i \zeta $  where  $\eta > 0$ and $\zeta \neq 0$. The zeros of the hypergeometric polynomials
 $$p_n(z)={}_{2}\text{F}_{1}\left[ \begin{array}{c} -n, \al n + 1, \\ \al n+2 \end{array} ;
 \begin{array}{cc} z \end{array}\right],$$ asymptotically cluster on the  loop of leminiscate
 \begin{equation*}\label{lami}
 |z^\eta(1-z)|e^{-\zeta Arg z}=\frac{|\al|^\eta}{{|\al+1|}^{\eta+1}}e^{\zeta Arg{\frac{\al}{\al+1}}}\text{ for }Re(z)>\frac{\eta}{\eta+1}.\end{equation*}

 \end{conjecture}

 \section{Some graphical examples and a general conjecture.}
 \begin{example} 
 \label{Ex2}The two pictures below show plots of the zeros of the hypergeometric polynomial
 $$p_n(z)={}_{2}\text{F}_{1}\left[ \begin{array}{c} -n, \al n + 1, \\ \al n+2 \end{array} ;
 \begin{array}{cc} z \end{array}\right]$$ and the loop of the lemniscate \eqref{lami} using Mathematica for different values of n and $\al$. The results support the preceding conjecture.
 \begin{figure}[H]
\begin{minipage}[b]{0.45\linewidth}
\centering
\includegraphics[width=\textwidth]{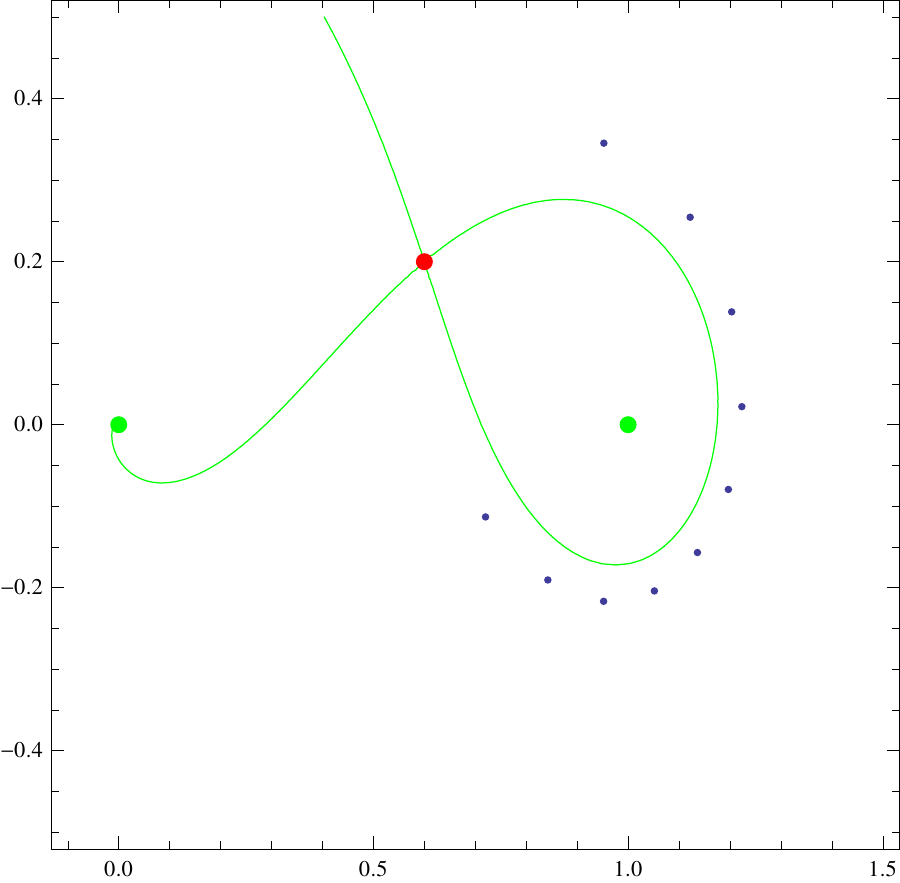}
\caption{$\al=\frac{1}{2}-i$ and n=10}
\label{fig:figure1}
\end{minipage}
\hspace{0.5cm}
\begin{minipage}[b]{0.45\linewidth}
\centering
\includegraphics[width=\textwidth]{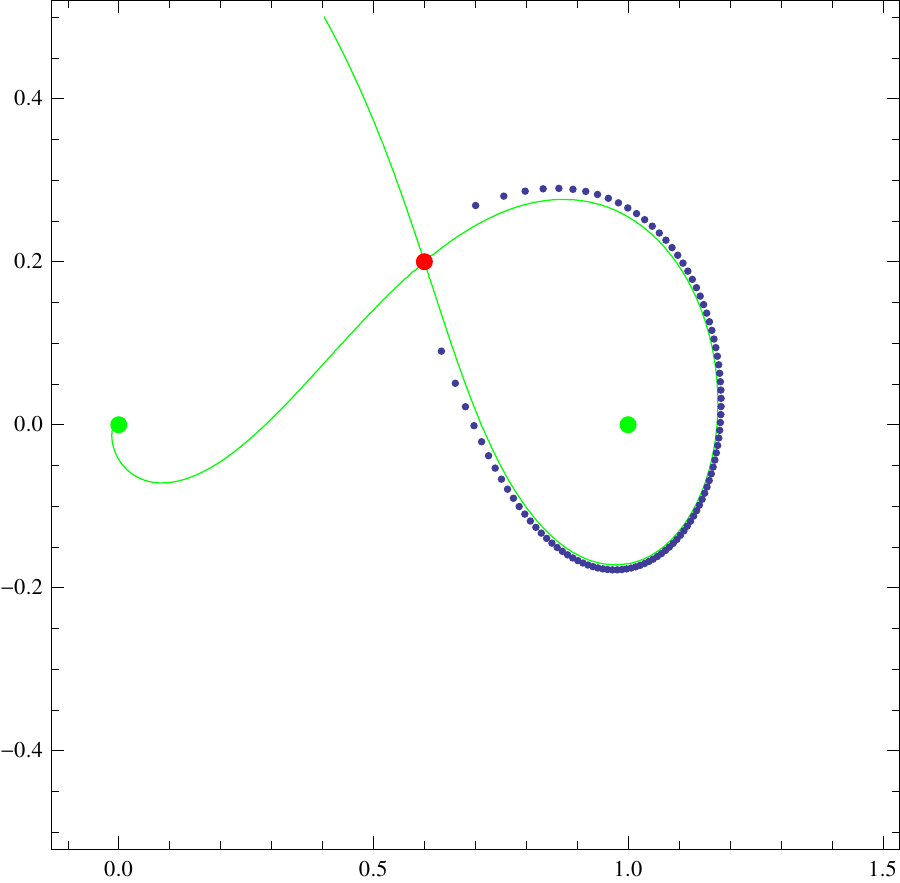}
\caption{$\al=\frac{1}{2}-i$ and $n=100$}
\label{fig:figure2}
\end{minipage}
\begin{minipage}[b]{0.45\linewidth}
\centering
\includegraphics[width=\textwidth]{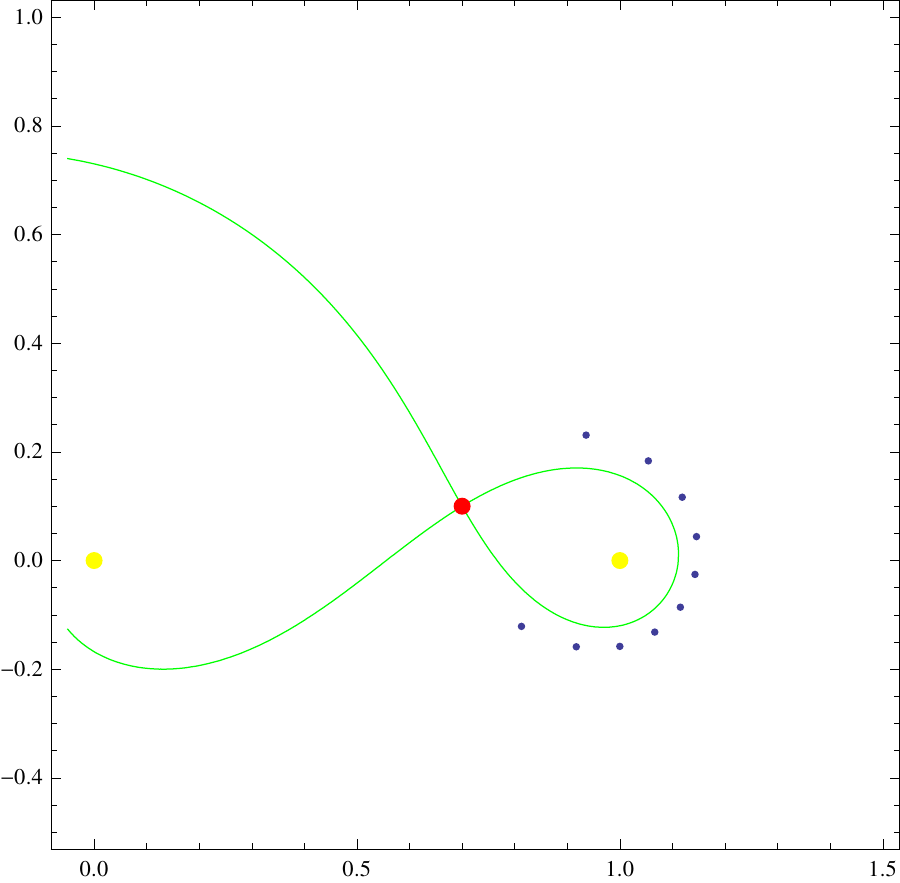}
\caption{$\al=2+i$ and n=10}
\label{fig:figure1}
\end{minipage}
\hspace{0.5cm}
\begin{minipage}[b]{0.45\linewidth}
\centering
\includegraphics[width=\textwidth]{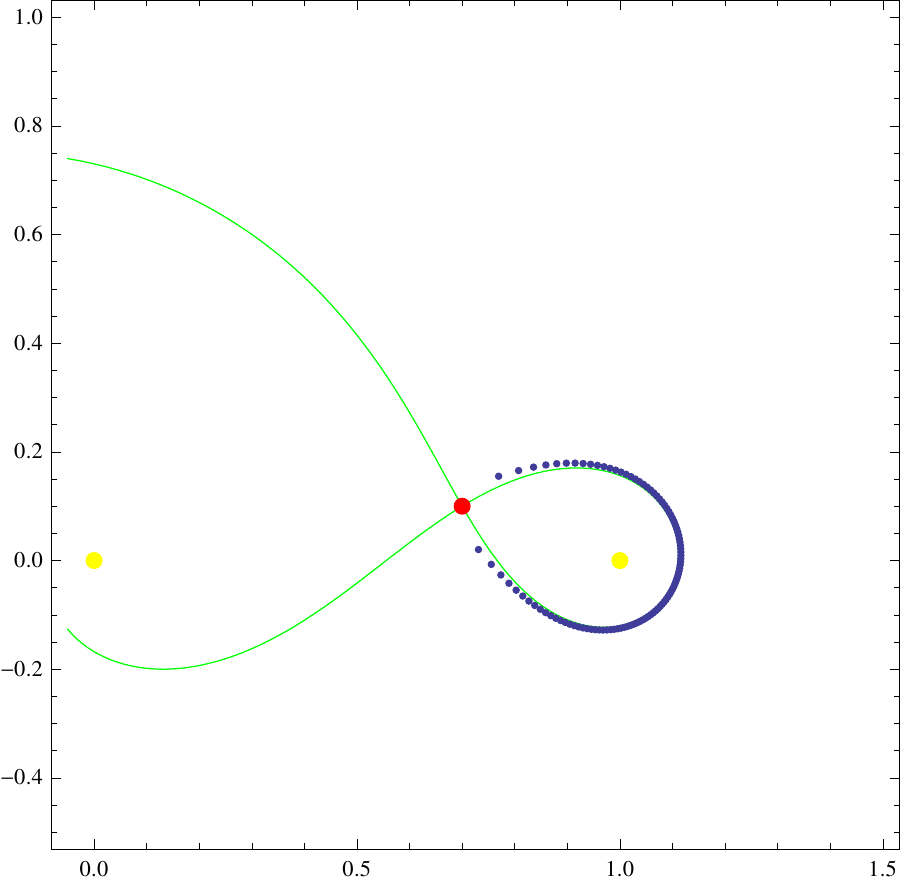}
\caption{$\al= 2+i$ and $n=100$}
\label{fig:figure2}
\end{minipage}
\end{figure}
 \end{example}
In the preceding pictures the two dots on the real axis are $0$ and $1$. Inside the lemniscate lies the point $1$. Given that $\mu$ has support on the lemniscate, we can use the theory in \cite{BB} to see that we must have $C(\mu)=-\alpha/z$ inside the lemniscate and $C(\mu)=1/(z-1)$ outside.  This follows since the logarithmic transform 
 $L(\mu)$ given by convolution of the asymptotic measure with $\log\vert z\vert$, is a subharmonic function, with a derivative $\frac{\partial L(\mu)}{\partial z}=C(\mu)$ that satisfies an algebraic equation. Thus locally in small disks around points on the lemniscate it is given by the maximum of $H_1$ and $\tilde H_2$(See Theorem 1 in \cite{BB}). But inside the lemniscate $\tilde H_2$ is the larger function, and the result on the Cauchy transform follows(in particular in the case of Example 1). Now note that we then actually had that the experimental data seems to indicate that there exists a region $D$ containing the support of $\mu$, such that $L(\mu)=Max\{ H_1,\tilde H_2\}$ in $D$. So the equality with the maximum of holds globally too, not only locally. In the next example we will first see that zeroes seem to cluster along level curves $H_1=\tilde H_i$ for some higher hypergeometric polynomials.

\begin{example}In this example we illustrate the level curves $H_1=\tilde{H_i}(z),\quad i=1,2$  and the zeros of
 $$p_n(z)={}_{3}\text{F}_{2}\left[ \begin{array}{c} -n, \al_1 n ,\al_2 n \\ \al_1 n+1,\al_2 n+1 \end{array} ;
 \begin{array}{cc} z \end{array}\right],$$  where  $Re(\al_i) >0$  for $i=1,2$ and $Im\al_1 \neq Im\al_2.$ Clearly the zeroes tend to cluster along parts of the level curves. The big dots are the branch points of the corresponding curve (and the singular points $0,1$).
 \begin{figure}[h!]
\begin{minipage}[b]{0.45\linewidth}
\centering
\includegraphics[width=\textwidth]{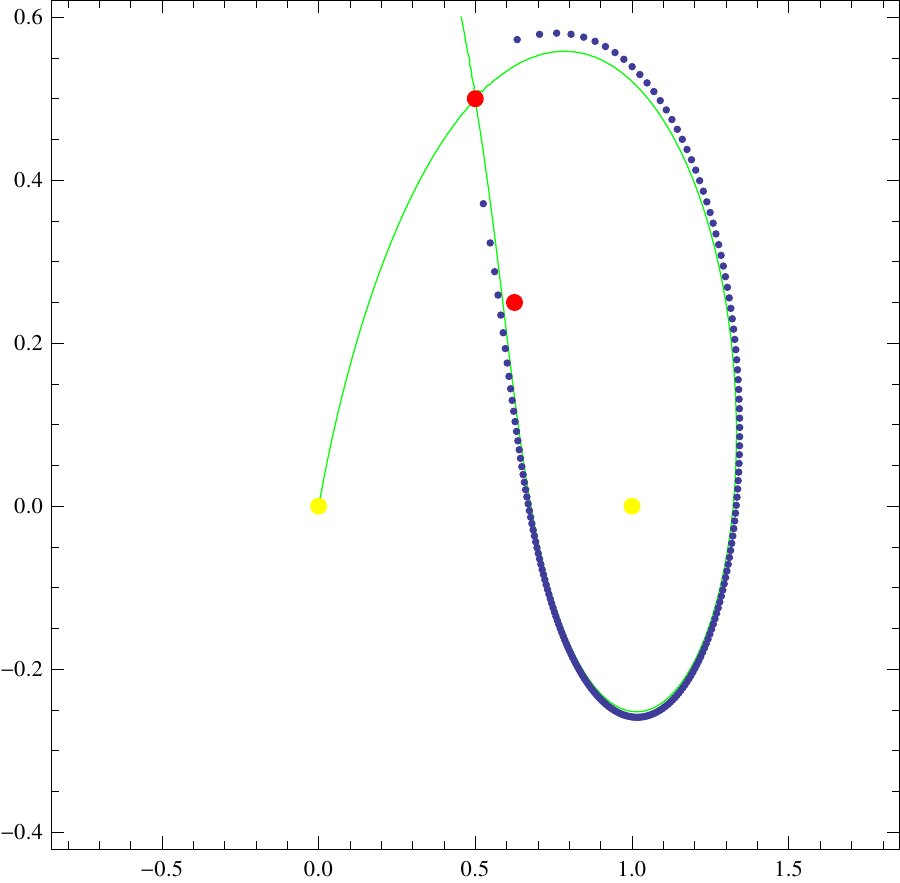}
\caption{$\al_1=i,\al_2=1+2i$ for $n=100$}
\label{fig:figure1}
\end{minipage}
\hspace{0.5cm}
\begin{minipage}[b]{0.45\linewidth}
\centering
\includegraphics[width=\textwidth]{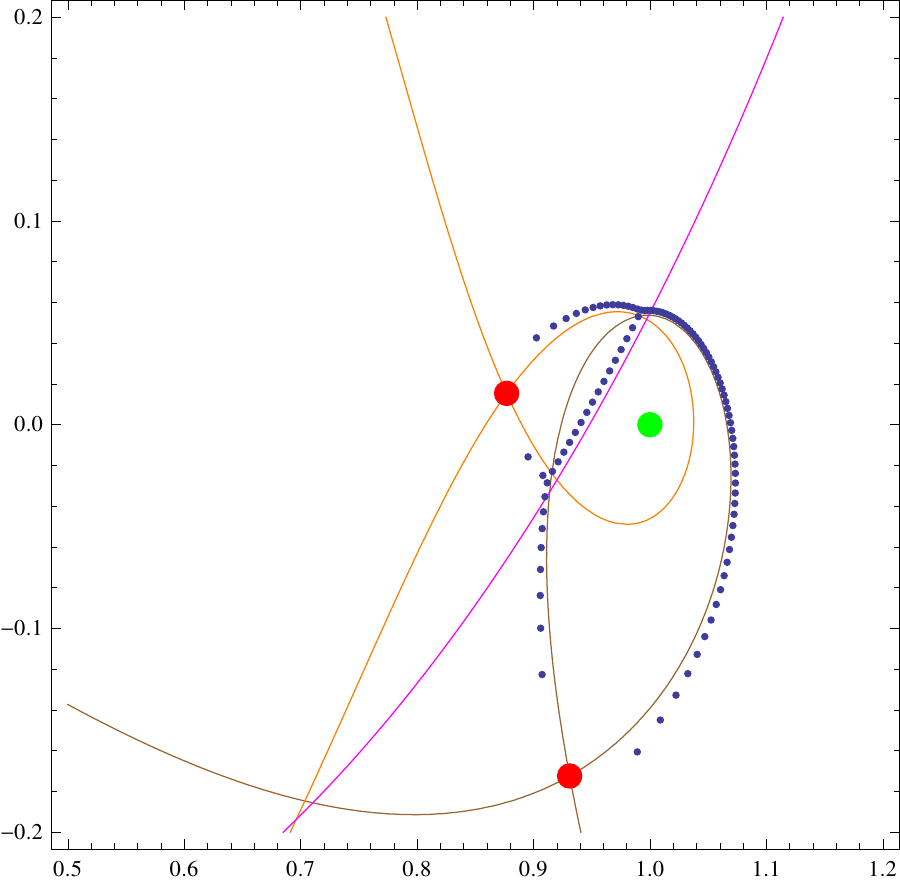}
\caption{$\al_1=1-5i,\al_2=7+i$ for $n=100$}
\label{fig:figure2}
\end{minipage}
\hspace{0.5cm}
\begin{minipage}[b]{0.45\linewidth}
\centering
\includegraphics[width=\textwidth]{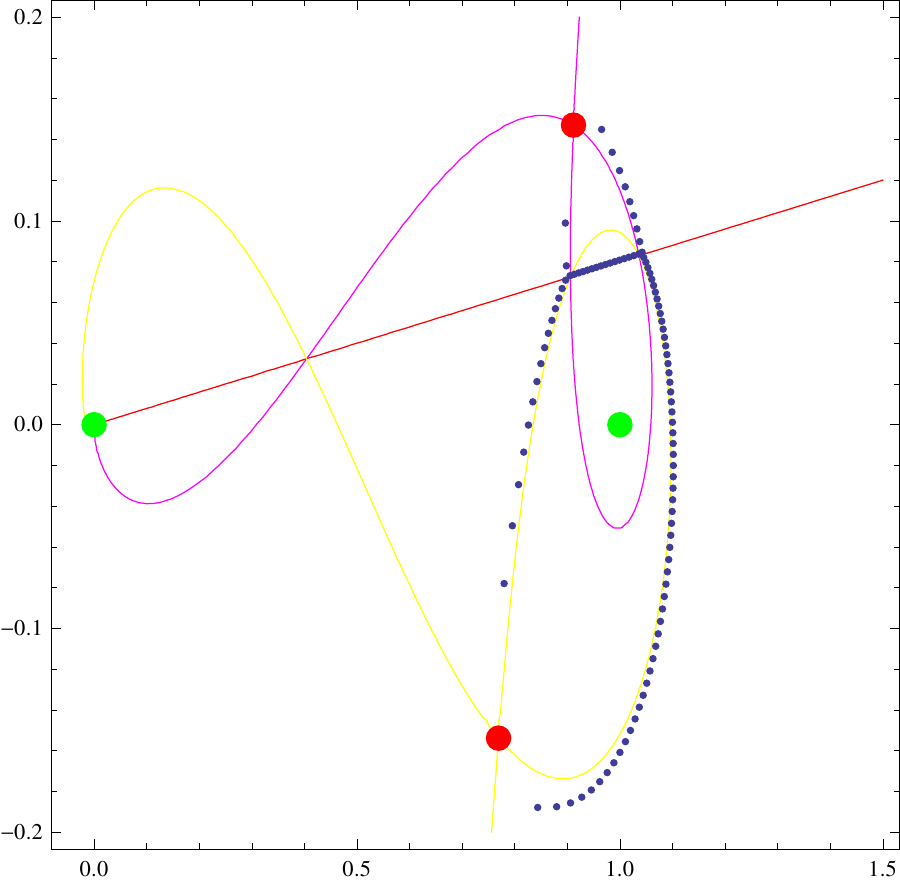}
\caption{$\al_1=2+5i,\al_2=2-2i$ for $n=100$}
\label{fig:figure3}
\end{minipage}

\end{figure}

 \end{example}

Define
$\varPsi(z)=Max\{\tilde{H_i}\}.$ It is a subharmonic, piecewise harmonic function, that is there exists $M_i$ that are open subset of $\bC$ who covers a.e such that $$\varPsi(z)=\tilde{H_i}(z) \quad\text{ on } \quad M_i.$$
Denote by $K$ the support of $\triangle \varPsi$. By piecewise harmonicity $K$ will be contained in the union of the level curves $\tilde{H_i}=\tilde{H_j},\ i\neq j.$ Our conjecture says that a subset of $K$ is the support of any asymptotic measure $\mu$, that comes from hypergeometric polynomials satisfying the conditions in Proposition \ref{prop3}. This would be a global result, saying which level curves will occur, in contrast to Theorem   \ref{levelsets}, which is a local statement, and only specifies that the support will be along some level curves.

  \begin{conjecture}There exists a connected and simply-connected domain  $D\subset \bC$, such that $K' :=K\cap D$ is the support of $\mu$, and $L_\mu=\varPsi$ in $D$. Moreover, $L_\mu=H_1$ outside $D$.
\end{conjecture}

 The pictures below (that have the same parameters as the last two figures) illustrate this conjecture.
  There we have first drawn all level curves $H_1=\tilde H_2$, $H_1=\tilde H_3$ and $\tilde H_3=\tilde H_2$. Then the big points that are the branch points.  Then we have marked the regions where $\varPsi(z)$
 is equal to $H_1,\tilde H_2,\tilde H_3$ respectively(writing $\tilde H_3=H_3+C_3$, etc.) This means that $K$ is visible as the boundary, where we change from one color coded harmonic function $H_i$ to another. We see that in both pictures the zeroes cluster along a subset of $K$, which is what the conjecture predicts and so this gives some experimental validation. The description of $L_\mu$ in the conjecture would actually follow from subharmonicity, as in example \ref{Ex2}, if the zeroes continue to behave as in the picture for large $n$.
 \begin{figure}[h!]
\begin{minipage}[b]{0.45\linewidth}
\centering
\includegraphics[width=\textwidth]{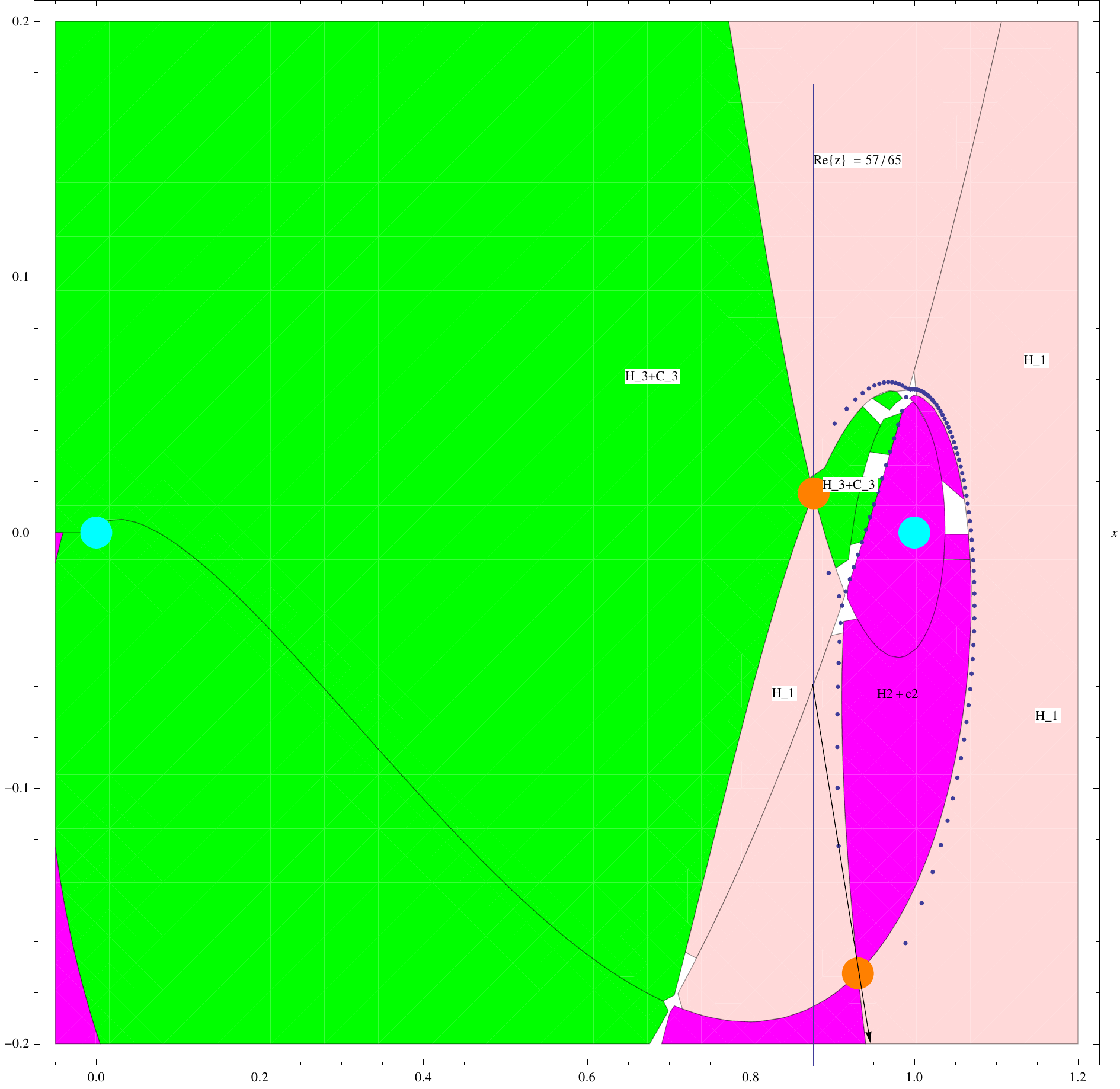}
\caption{Region plot of $\psi(z)$ and the zeros of $ p_n$ for $n=100$}
\label{fig:figure2}
\end{minipage}
\hspace{0.5cm}
\begin{minipage}[b]{0.45\linewidth}
\centering
\includegraphics[width=\textwidth]{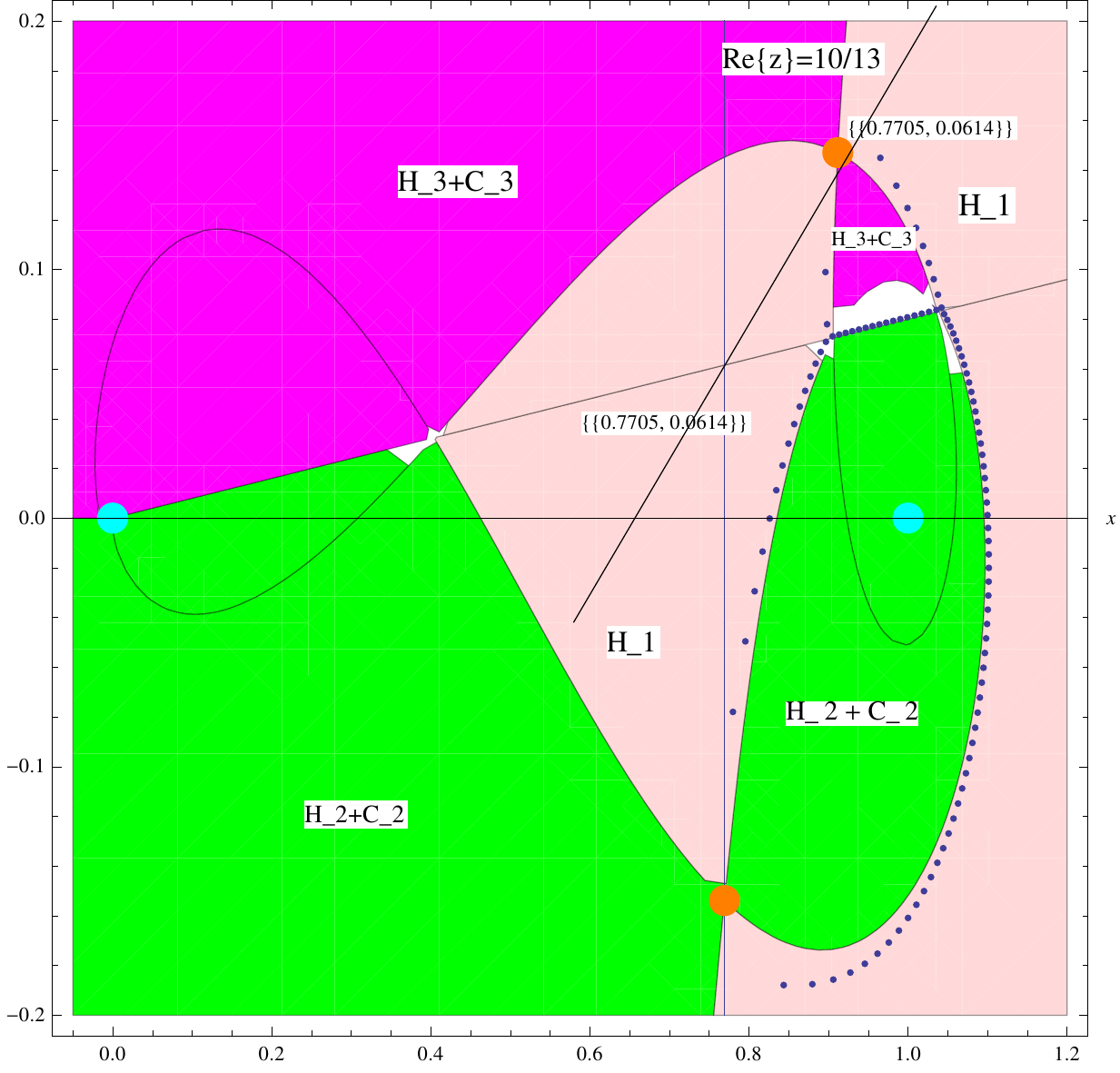}
\caption{Region plot of $\psi(z)$ and the zeros of $ p_n$ for $n=100$ }
\label{fig:figure3}
\end{minipage}
\end{figure}
\clearpage

\bibliographystyle{plain}
\bibliography{xx}

  \end{document}